\documentclass{amsart}
\usepackage{amsmath,amssymb,epsfig,amscd,amsthm}
\usepackage[margin=1.5in]{geometry}
\usepackage{tikz}
\usetikzlibrary{arrows}
\usetikzlibrary{matrix}
\usepackage{graphicx}

\usepackage{verbatim}

\numberwithin{equation}{section}

\newtheorem{lemma}[equation]{Lemma}
\newtheorem{thm}[equation]{Theorem}
\newtheorem{conjecture}[equation]{Conjecture}
\newtheorem{cor}[equation]{Corollary}
\newtheorem{prop}[equation]{Proposition}

\newtheorem{question}[equation]{Question}

\newtheorem{defi}[equation]{Definition}

\theoremstyle{remark}

\renewcommand{\bar}[1]{#1\llap{$\overline{\phantom{\rm#1}}$}}

\newcommand{\lra}{\longrightarrow}

\newcommand{\bA}{{\mathbb A}}
\newcommand{\N}{{\mathbb N}}
\newcommand{\Z}{{\mathbb Z}}

\newcommand{\R}{{\mathbb R}}
\newcommand{\C}{{\mathbb C}}

\newcommand{\M}{{\mathcal{M}}}

\newcommand{\Qbar}{\bar{\mathbb{Q}}}

\renewcommand{\l}{\lambda}

\newcommand{\bff}{{\mathbf f}}
\newcommand{\bfg}{{\mathbf g}}

\newcommand{\cP}{\mathcal{P}}

\newcommand{\bP}{{\mathbb P}}

\newcommand{\cM}{\mathcal{M}}

\begin{document}



\title{A case of the Dynamical Andr\'e-Oort Conjecture}

\author{D.~Ghioca}
\address{
Dragos Ghioca\\
Department of Mathematics\\
University of British Columbia\\
Vancouver, BC V6T 1Z2\\
Canada
}
\email{dghioca@math.ubc.ca}

\author{H.~Krieger}
\address{
Holly Krieger\\
Department of Mathematics\\
Massachusetts Institute of Technology\\
77 Massachusetts Avenue\\
Cambridge, MA 02139\\
USA
}
\email{hkrieger@math.mit.edu}

\author{K.~Nguyen}
\address{
Khoa Nguyen
Department of Mathematics\\
University of California at Berkeley\\
Berkeley, CA 94720\\
USA}
\email{khoanguyen2511@berkeley.edu}

\keywords {Mandelbrot set, unlikely intersections in dynamics}
\subjclass[2010]{Primary 37F50; Secondary 37F05}
\thanks{The first author was partially supported by an NSERC Discovery Grant. The second author was partially supported by NSF grant DMS-1303770.}


 \begin{abstract}
We prove a special case of the Dynamical Andr\'e-Oort Conjecture formulated by Baker and DeMarco \cite{Matt-Laura}. For any integer $d\ge 2$, we show that for a rational plane curve $C$ parametrized  by $(t, h(t))$ for some non-constant polynomial $h\in\C[z]$, if there exist infinitely many points $(a,b)\in C(\C)$ such that both $z^d+a$ and $z^d+b$ are postcritically finite maps, then $h(z)=\xi z$ for a $(d-1)$-st root of unity $\xi$. As a by-product of our proof, we show that the Mandelbrot set is not the filled Julia set of any polynomial $g\in\C[z]$.
 \end{abstract}


\maketitle

\section{Introduction}
\label{intro}

Motivated by the results on unlikely intersection in arithmetic dynamics from \cite{Matt-Laura, prep, prep-2}, Baker and DeMarco \cite[Conjecture~1.4]{Matt-Laura-2} recently formulated a dynamical analogue (see Conjecture \ref{dynamical Andre-Oort conjecture}) of the Andr\'e-Oort Conjecture, characterizing the subvarieties of moduli spaces of dynamical systems which contain a Zariski dense subset of postcritically finite points.  A morphism $f:\bP^1\lra \bP^1$ of degree larger than $1$ is said to be \emph{postcritically finite (PCF)} if each of its critical points is preperiodic; that is, has finite forward orbit under iteration by $f$.  We prove  in Section~\ref{proof section} the following supporting case of the dynamical Andr\'e-Oort Conjecture.

\begin{thm}
\label{main result}
Let $d\ge 2$ be an integer, and let $h\in \C[z]$ be a non-constant polynomial. If there exist infinitely many $t\in \C$ such that both maps $z\mapsto z^d+t$ and $z\mapsto z^d+h(t)$ are PCF maps, then $h(z)=\zeta z$, where $\zeta$ is a $(d-1)$-st root of unity.
\end{thm}

As a by-product of our proof we show that the Mandelbrot set is not a filled Julia set for any polynomial $h\in\C[z]$; more generally, we show that no multibrot set is a filled Julia set.  The \emph{$d$-th multibrot set} $\mathcal{M}_d$ is the set of all $t\in\C$ with the property that the orbit of $0$ under the map $z\mapsto z^d+t$ is bounded (with respect to the usual archimedean absolute value); when $d=2$ we obtain the classical Mandelbrot set. The \emph{filled Julia set $K_h$} of a polynomial $h(z)$  is the set of all $z\in\C$ such that $|h^n(z)|$ is bounded independently of $n\in\N$, where $h^n(z)$ is the $n$-th iterate of the map $z\mapsto h(z)$. We prove:

\begin{thm} \label{MnotJ} For each $d\ge 2$, there does not exist a polynomial $h(z)$ whose filled Julia set is $\mathcal{M}_d$.
\end{thm}

Theorem~\ref{MnotJ} was widely believed to be true in the complex dynamics community, and intuitively \emph{seems} obvious for various reasons; for example, one expects a filled Julia set to have a stricter self-similarity than a multibrot set. To our knowledge, however, there is no published proof, nor there was known for specialists a \emph{clear} argument for proving Theorem~\ref{MnotJ}.  We provide two proofs for Theorem \ref{MnotJ}.  In Section~\ref{Mandelbrot section}, we use classical results in complex dynamics to deduce the result. In particular, we prove precise results regarding the landing pairs of parameter rays on the $d$-th multibrot set (see Proposition~\ref{sector}). In the proof of Theorem~\ref{MnotJ} from Section~\ref{Mandelbrot section} we also use an old result of Bang \cite{Bang} regarding the existence of primitive prime divisors in arithmetic sequences. In Section~\ref{algebraic}, we provide a primarily algebraic proof of Theorem~\ref{MnotJ}, which is shorter, but relies on the deep no-wandering-domains theorem of Sullivan \cite{Sullivan} and the delicate classification theory of invariant curves of $\mathbb{P}^1 \times \mathbb{P}^1$ done by Medvedev and Scanlon \cite{Medvedev-Scanlon}. In Theorem~\ref{thm:1st} we prove that if a polynomial $g(z)$ has a Fatou component which is the image of the open unit disk under another polynomial map, then the Julia set of $g$ is a circle, and so $g(z)=az^k+b$ for some $a,b\in\C$ (according to \cite{Schmidt-Steinmetz}). Since the main hyperbolic component of $\M_d$ is easily seen to be the image of the open unit disk under a polynomial map, we obtain the conclusion of Theorem~\ref{MnotJ}. We also discuss in Section~\ref{algebraic} additional related problems, and we state a general question regarding possible algebraic relations between the Fatou components (respectively between the Julia sets) of two polynomials (see Question~\ref{q:1st}).  Given the contrasting approaches and backgrounds required for our two proofs, and also given the intrinsic interests in both approaches, we include both for the benefit of the reader.

The connection between Theorem~\ref{main result} and Theorem~\ref{MnotJ} is made in Section~\ref{proof section} through a result on unlikely intersections in dynamics proven in \cite[Theorem~1.1]{prep-2}. Namely, we use \cite{prep-2} to show that under the assumptions of Theorem \ref{main result}, the set
$$\cP:=\{t\in\C\colon z^d+t \text{ is a PCF map}\}$$
is totally invariant under the polynomial $h$.  When $\text{deg}(h) \geq 2,$ we deduce a contradiction to Theorem \ref{MnotJ}; if $\text{deg}(h) = 1$, a simple complex-analytic argument yields $h(t) = \zeta t$ for some $(d-1)$-st root of unity $\zeta$ (see Proposition~\ref{affinesymm}).

Motivating our study is a common theme for many outstanding conjectures in arithmetic geometry: given a quasiprojective variety $X$, one defines the concepts of \emph{special points} and \emph{special subvarieties} of $X$. Then the expectation is that whenever $Y\subseteq X$ contains a Zariski dense set of special points, the subvariety $Y$ must be itself special. For example, if $X$ is an abelian variety and we define as \emph{special points} the torsion points of $X$, while we define as \emph{special subvarieties} the torsion translates of abelian subvarieties of $X$, we obtain the Manin-Mumford conjecture (proved by Raynaud \cite{Ray1, Ray2}).

The Andr\'e-Oort Conjecture follows the same principle outlined above. We present the conjecture in the case when the ambient space is the affine plane. In this case, the conjecture was proven first by  Andr\'e \cite{Andre} and then later generalized by Pila \cite{Pila} to mixed Shimura varieties (in which the ambient space is a product of curves which are either a modular curve, an elliptic curve, or the multiplicative group). 

\begin{thm}[Andr\'e \cite{Andre}, Pila \cite{Pila}]
\label{Pila's theorem}
Let $C\subset \bA^2$ be a complex plane curve containing infinitely many points $(a,b)$ with the property that the elliptic curves with $j$-invariant equal to $a$, respectively $b$ are both CM elliptic curves. Then $C$ is either horizontal, or vertical, or it is the zero set of a modular polynomial. 
\end{thm}

Motivated by the Andr\'e-Oort Conjecture, Baker and DeMarco \cite{Matt-Laura-2} formulated a dynamical version of the Andr\'e-Oort Conjecture where the ambient space is the moduli space of rational maps of degree $d$, and the special points are represented by the PCF maps. The postcritically finite (PCF) maps are crucial in the study of dynamics, but also analogous to CM points in various ways. The PCF points form a countable, Zariski dense subset in the dynamical moduli space, and over a given number field, there are only finitely many non-exceptional PCF maps of fixed degree (see \cite{BIJL} for proof, and definition of non-exceptional).  Additionally, Jones \cite{Jones} has shown that the arboreal Galois representation associated to a rational map defined over a number field $K$ will have small (of infinite index) image for PCF maps despite the image being generally of finite index, analogous to the situation for the $\ell$-adic Galois representation associated to an elliptic curve over $K$, according to whether the curve has CM or not.  We state below a special case of  \cite[Conjecture~1.4]{Matt-Laura-2}, which is much closer in spirit to Theorem~\ref{Pila's theorem} and avoids some of the technical assumptions stated in \cite[Conjecture~1.4]{Matt-Laura-2}. 

\begin{conjecture}
\label{dynamical Andre-Oort conjecture}
Let $C\subset \bA^2$ be a complex plane curve with the property that it contains infinitely many points $(a,b)$ with the property that both $z^2+a$ and $z^2+b$ are PCF maps. Then $C$ is either  horizontal, or vertical, or it is the diagonal map.
\end{conjecture}

The Andr\'e-Oort Conjecture fits also into another more general philosophy common for several major problems in arithmetic geometry which says that each \emph{unlikely (arithmetic)  intersection} occuring more often than expected must be explained by a geometric principle. Typical for this principle of unlikely intersections is the Pink-Zilber Conjecture (which is, in turn, a generalization of the Manin-Mumford Conjecture). Essentially, one expects that the intersection of a subvariety $V$ of a semiabelian variety $X$  with the union of all algebraic subgroups of $X$ of codimension larger than $\dim(V)$ is not Zariski dense in $V$, unless $V$ is contained in a proper algebraic subgroup of $X$; for more details on the Pink-Zilber Conjecture, see the beautiful book of Zannier \cite{Zannier-book}. Taking this approach, Masser and Zannier \cite{M-Z-1, M-Z-2} proved that given any two sections $S_1, S_2:\bP^1\lra E$ on an elliptic surface $\pi:E\lra \bP^1$, if there exist infinitely many $\l\in \bP^1$ such that both $S_1(\l)$ and $S_2(\l)$ are torsion on the elliptic fiber $E_\l:=\pi^{-1}(\l)$, then $S_1$ and $S_2$ are  linearly dependent over $\Z$ as points on the generic fiber of $\pi$. For a general conjecture extending both the Pink-Zilber and the Andr\'e-Oort conjectures, see \cite{Pink}.

At the suggestion of Zannier, Baker and DeMarco \cite{Matt-Laura} studied a dynamical analogue of the above problem of simultaneous torsion in families of elliptic curves. The main result of \cite{Matt-Laura} is to show that given two complex numbers $a$ and $b$, and an integer $d\ge 2$, if there exist infinitely many $\l\in\C$ such that both $a$ and $b$ are preperiodic under the action of $z\mapsto z^d+\l$, then $a^d=b^d$. The result of Baker and DeMarco may be viewed as the first instance of the unlikely intersection problem in arithmetic dynamics. Later more general results were proven for arbitrary families of polynomials \cite{prep} and for families of rational maps \cite{prep-2}. The proofs from \cite{Matt-Laura, prep, prep-2} use powerful equidistribution results for points of small height (see \cite{Baker-Rumely06, CL, favre-rivera06, Yuan, YZ}). Conjecture~\ref{dynamical Andre-Oort conjecture} is yet another statement of the principle of unlikely intersections in algebraic dynamics.
 
Theorem~\ref{main result} yields that Conjecture~\ref{dynamical Andre-Oort conjecture} holds for all plane curves of the form $y=h(x)$ (or $x=h(y)$) where $h\in \C[z]$. One may attempt to attack the general Conjecture~\ref{dynamical Andre-Oort conjecture} along the same lines. However, there are significant difficulties to overcome. First of all, it is not easy to generalize the unlikely intersection result of \cite[Theorem~1.1]{prep-2} (see also Theorem~\ref{GHT}). This is connected with another deep problem in arithmetic geometry which asks how smooth is the variation of the canonical height of a point across the fibers of an algebraic family of dynamical systems (see \cite[Section~5]{prep-2} for a discussion of this connection to the works of Tate \cite{Tate} and Silverman \cite{Silverman83, silverman94-1, silverman94-2}). It is also unclear how to generalize Theorem~\ref{MnotJ}; in Section~\ref{algebraic} we discuss some possible extensions of Theorem~\ref{MnotJ} (see Question~\ref{MnotJ conjecture}).

\medskip

\emph{Acknowledgments.}  We thank the American Institute of Mathematics (AIM) and the  organizers of the workshop ``Postcritically finite maps in complex and arithmetic dynamics'' hosted by the AIM where the collaboration for this project began. We are grateful to Patrick Ingram who asked Conjecture \ref{dynamical Andre-Oort conjecture} during the aforementioned workshop. We thank Thomas Tucker for several useful discussions regarding this project. Last, but definitely not least, we are indebted to  Laura DeMarco for her very careful reading of a previous version of our paper and for suggesting numerous improvements.

\section{Proof of Theorem~\ref{main result}}
\label{proof section}

In this Section, we deduce Theorem~\ref{main result} from Theorem~\ref{MnotJ} combined with two other results. The first of the two results we need (see Theorem~\ref{GHT}) is proven in \cite[Theorem~1.1]{prep-2} in higher generality than the one we state here, and it is a consequence of the powerful equidistribution results of Yuan and Zhang \cite{Yuan, YZ} combined with a complicated analysis of the variation of the canonical height in an algebraic family of morphisms. The result we will use in the proof of Theorem~\ref{main result} is \cite[Theorem~1.1]{prep-2} stated for   polynomial families over the base curve $\bP^1$.

\begin{thm} \label{GHT} Let $f\in\Qbar[z]$ be a polynomial of degree $d\ge 2$, and let $P,Q\in\Qbar[z]$ be nonconstant polynomials. We let $\bff_t(z):=f(z)+P(t)$ and $\bfg_t(z):=f(z)+Q(t)$ be two families of polynomial mappings, and we let $a,b\in\Qbar$. If there exist infinitely many $t\in\Qbar$ such that both $a$ is preperiodic for $\bff_t$ and $b$ is preperiodic for $\bfg_t$, then for each $t\in\Qbar$ we have that $a$ is preperiodic for $\bff_t$ if and only if $b$ is preperiodic for $\bfg_t$.
\end{thm}

The second result we need for the proof of Theorem~\ref{main result} is the result below,  proved in Section~\ref{Mandelbrot section} by analyzing the coefficients of the analytic  isomorphism constructed in \cite{Douady-Hubbard} between the complement of $\M_d$ and the complement of the closed unit disk.

\begin{prop} \label{affinesymm} Suppose $\mu(z) = Az +B$ is an affine symmetry of $\mathbb{C}$ satisfying $\mu(\mathcal{M}_d) = \mathcal{M}_d$.  Then $A=\xi$ and $B=0$, where $\xi$ is  a $(d-1)$-st root of unity; moreover, all $(d-1)$-st roots of unity provide a rotational symmetry of $\mathcal{M}_d$.
\end{prop}

Using Theorems~\ref{MnotJ} and \ref{GHT}, and Proposition~\ref{affinesymm} we can prove now Theorem~\ref{main result}. In terms of notation, we always denote by $\partial S$ the boundary of a set $S\subseteq \C$.

\begin{proof}[Proof of Theorem \ref{main result}] Fix $d \geq 2,$ and suppose $h(z) \in \mathbb{C}[z]$ is a non-constant polynomial so that both $\bff_t(z) = z^d+t$ and $\bfg_t(z) = z^d+h(t)$ are PCF for infinitely many $t \in \mathbb{C}$. First of all, it is immediate to deduce that each parameter $t$ such that $\bff_t$ is PCF (i.e., $0$ is preperiodic for $z\mapsto z^d+t$) is an algebraic number. Then  Theorem~\ref{GHT} yields that for all $t \in \mathbb{C}$, the map $z\mapsto \bff_t(z)$ is PCF if and only $z\mapsto\bfg_t(z)$ is PCF.  

The closure of the set of parameters $t \in\mathcal{M}_d$ which correspond to a PCF map $z\mapsto z^d+t$ contains the boundary of $\mathcal{M}_d$;  more precisely, this boundary can be topologically identified by the PCF points as follows: $\partial \mathcal{M}_d$ consists exactly of those $c \in \mathbb{C}$ such that every open neighborhood of $c$ contains infinitely many distinct PCF parameters.  Using this characterization of $\partial \mathcal{M}_d$ and the above conclusion of Theorem~\ref{GHT} that $\bff_t(z)$ is PCF iff $\bfg_t(z)$ is PCF, we see that by continuity of $h$ and the open mapping theorem, $\partial \mathcal{M}_d$ is totally invariant for $h$; that is, $ \partial \mathcal{M}_d = h^{-1}(\partial \mathcal{M}_d)$. Since $\mathbb{C} \setminus \mathcal{M}_d$ is connected, and $h$ maps with full degree on a neighborhood of $\infty$, $\mathbb{C} \setminus \mathcal{M}_d$ is also totally invariant for $h$.  So, $h^{-1}(\M_d)=\M_d$, and if $h$ is linear, then $h$ is an affine symmetry of $\mathcal{M}_d$; in this case  Proposition \ref{affinesymm} yields the desired conclusion.

Suppose from now on that $h$ has degree at least $2$, and denote by $h^n$ the $n$-th iterate of $h$ under composition.  Recall the definition of the \emph{filled Julia set of $h$} 
$$K_h := \{ z \in \mathbb{C} : |h^n(z)| \text{ is bounded uniformly in } n \},$$
and then the \emph{Julia set of $h$} is $J_h := \partial K_h$; by Montel's theorem (see \cite{Milnor}), $J_h$ is also the minimal closed set containing at least 3 points which is totally invariant for $h$.  Therefore $J_h \subset \partial \mathcal{M}_d$.  Since $\mathbb{C} \setminus \mathcal{M}_d$ is connected, $K_h \subset \mathcal{M}_d$.   On the other hand, since $\mathbb{C} \setminus \mathcal{M}_d$ is totally invariant for $h$ and contains a neighborhood of $\infty$, every point of $\mathcal{M}_d$ is bounded under iteration by $h$, and so by definition, $\mathcal{M}_d \subset K_h$.  Therefore $\mathcal{M}_d = K_h$, contradicting Theorem \ref{MnotJ}.
\end{proof}

\section{The $d$-th multibrot set is not a filled Julia set}
\label{Mandelbrot section}

Our goal in the Section is to provide a complex-dynamical proof of Theorem \ref{MnotJ}.  We first recall a few basic facts and definitions from complex dynamics; see \cite{Douady-Hubbard, Milnor}. Throughout this section we denote by $\overline{\mathbb{D}}$ the closed unit disk in the complex plane, and we denote by $S^1$ its boundary (i.e., the unit circle in the complex plane).   Let $f: \mathbb{C} \rightarrow \mathbb{C}$ be a polynomial of degree $d \geq 2$.  If the filled Julia set $K_f$ is connected, there exists an isomorphism (called a B\"ottcher coordinate of $f$)
$$\phi_f: \mathbb{C} \setminus K_f \rightarrow \mathbb{C} \setminus \overline{\mathbb{D}}$$
so that $\phi_f$ conjugates $f$ to the $d$-th powering map. Writing 
$$\alpha = \lim_{z \rightarrow \infty} \frac{\phi_f(z)}{z},$$ 
one can show that $\alpha$ is a $(d-1)$-st root of the leading coefficient of $f$, and $\phi_f$ is unique up to choice of this root. When $f$ is monic, we always
normalize $\phi_f$ so that $\alpha=1$.

Fix $d \geq 2$, and let $f_c(z) := z^d+c$ (with the exception of the $d$-th multibrot set, we will suppress dependence on $d$ in notation). We note that  the $d$-th multibrot set $\M_d$ is defined to be the set of parameters $c \in \mathbb{C}$ such that the  Julia set $J_{f_c}$ is connected; equivalently, $\mathcal{M}_d$ is the set of parameters for which the forward orbit of $0$ under $f_c$ is bounded.  As described in \cite{Douady-Hubbard}, if $c \not\in \mathcal{M}_d$, then the B\"ottcher coordinate $\phi_{f_c}$ will not extend to the full $\mathbb{C} \setminus K_{f_c},$ but is guaranteed to extend to those $z$ which satisfy $G_{f_c}(z) > G_{f_c}(0)$, where $G_{f_c}$ is the dynamical Green's function 
$$G_{f_c}(z) = \lim_{n \rightarrow \infty} \frac{\log|f_c^n(z)|}{d^n}.$$ 
In particular, the B\"ottcher coordinate is defined on the critical value $c$.  Therefore we have a function $\Phi(c) := \phi_{f_c}(c)$, 
and one can show (see also \cite{Matt-Laura}) that this function is an analytic isomorphism $\mathbb{C} \setminus \mathcal{M}_d \rightarrow \mathbb{C} \setminus \overline{\mathbb{D}}$, with 
$$\lim_{c \rightarrow \infty} \frac{\Phi(c)}{c} =1.$$  
Using a careful analysis of the coefficient of $\Phi$ allows us to compute the affine symmetry group of $\mathcal{M}_d$ and thus prove Proposition~\ref{affinesymm}. 

\begin{proof}[Proof of Proposition~\ref{affinesymm}.] 
By definition of the $d$-th multibrot set, it is clear that $\M_d$ is invariant under rotation by an angle multiple of $\frac{2\pi}{d-1}$.  Suppose now that $\mu(z)=Az+B$ fixes $\mathcal{M}_d$; it follows that $\mu$ fixes $\mathbb{C} \setminus \mathcal{M}$.  Denote by $\Psi(z)$ the inverse of $\Phi(z)$.  The automorphism $m(z) = \Phi \circ \mu \circ \Psi$ fixes $\infty$ and so is a rotation $m(z) = \lambda z$ with  $\lambda \in S^1$. Now,  $\Psi$ has local expansion about $\infty$:
$$\Psi(z) = z + \sum_{m=0}^{\infty} b_m z^{-m};$$
as computed by Shimauchi in \cite{Shimauchi}, where $b_m = 0$ for $0 \leq m < d-2$ and $b_{d-2} \ne 0$. Expanding the equality $\Psi(m(z)) = \mu(\Psi(z))$ shows $A = \lambda$ is a $(d-1)$-st root of unity; for $d >2$ we also have $B = 0$ as desired.  For $d = 2$, one computes $b_{0} = -1/2$ and $b_1 \ne 0$, and again expanding we see that $A = \lambda = \pm 1$ and $B = \lambda/2 - 1/2.$  Since $z \mapsto -z-1$ is clearly not a symmetry of the Mandelbrot set, then there are no nontrivial affine symmetries of $\M_2$ and so, Proposition~\ref{affinesymm} is proved.
\end{proof}

From now on we fix a polynomial $h(z)$ with complex coefficients having connected Julia set. The goal of Theorem~\ref{MnotJ} is to show that the filled Julia set of $h$ does not equal $\mathcal{M}_d$. As before, we let $\phi_h$ be a B\"ottcher coordinate for $h$.

\begin{defi} Fix $\theta \in [0, 1]$.  The external ray 
$$\mathcal{R}_h(\theta) := \phi_h^{-1} \left( \{ re^{2 \pi i \theta} : r > 1 \} \right)$$
is the {\bf dynamic ray} corresponding to $\theta$; the external ray 
$$\mathcal{R}(\theta) := \Phi^{-1} \left( \{ re^{2 \pi i \theta} : r > 1 \} \right)$$
is the {\bf parameter ray} corresponding to $\theta$. Sometimes, by abuse of language, we will refer to a (dynamic or parameter) ray simply by the corresponding angle $\theta$. 
\end{defi}

The following result is proved in \cite{Douady-Hubbard}.
\begin{thm}[Douady-Hubbard]
\label{rays land} 
All parameter rays $\mathcal{R}(\theta)$ with $\theta \in \mathbb{Q}$ will land; that is, the limit $\lim_{r \rightarrow 1} \Phi(r e^{2 \pi i \theta})$ exists and lies on the boundary of the $d$-th multibrot set.  If $\theta$ is rational with denominator coprime to $d$, then there is at most one $\theta'  \in [0, 1]$ such that $\theta' \ne \theta$, and $\mathcal{R}(\theta)$ and $\mathcal{R}(\theta')$ land at the same point. 
\end{thm}

We call such a pair $(\theta, \theta')$ a {\it landing pair}.   The {\it period} of the pair is the period of $\theta$ (and $\theta'$) under the multiplication-by-$d$ map modulo 1; call this map $\tau_d$.  Any pair of period $n$ will land at the root of a hyperbolic component of period $n$.  Note that by definition, any $\theta$ of period $n$ under $\tau_d$ will have denominator dividing $d^n - 1$. As an important example of the above facts, the only period $2$ angles under the doubling map $\tau_2$ are $\theta = 1/3$ and $\theta' = 2/3$, which land at the root point of the unique period 2 hyperbolic component, namely at $c = -\frac{3}{4}$ (see Figure 1).  

\begin{figure}
  \caption{Examples of parameter rays for $d = 2$.}
  \centering
    \includegraphics[width=0.5\textwidth]{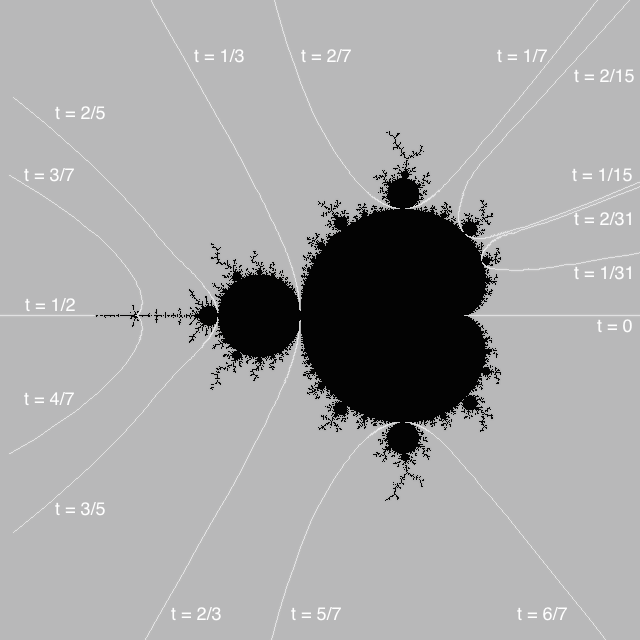}
\end{figure}

We make the obvious but important remark that since $h$ is (continuously) defined on the entire complex plane, any pair of dynamic rays $\mathcal{R}_h(\theta)$ and $\mathcal{R}_h(\theta')$ which land on the same point will be mapped under $h$ to a (possibly equal) pair of rays that also land together.  The impossibility of such a map on the parameter rays of the $d$-th multibrot set which preserves landing pairs is the heart of the proof of  Theorem \ref{MnotJ}.  

Finally, note that external rays cannot intersect, and that removing any landing pair $\mathcal{R}_h(\theta), \mathcal{R}_h(\theta')$ will decompose $\mathbb{C} \setminus K_h$ into a union of two disjoint open sets.  Thus we can make the following definition.

\begin{defi} Suppose $\mathcal{R}_h(\theta)$ and $\mathcal{R}_h(\theta')$ land at the same point. The open set in $\mathbb{C} \setminus (K_h \cup \mathcal{R}_h(\theta) \cup \mathcal{R}_h(\theta'))$ which does not contain the origin is the {\bf wake} of the rays. 
\end{defi}

We recall the following classical  classification of fixed points from complex dynamics (for more details, see \cite{Milnor}).

\begin{defi} \label{basicdyn} If $f_c(\alpha) = \alpha$, we call $\lambda = f_c'(\alpha)$ the {\bf multiplier} of the fixed point. If $|\lambda|<1$, then we say that $\alpha$ is {\bf attracting}.  If $\lambda$ is a root of unity, we say the fixed point is {\bf parabolic}.
\end{defi}

Denote by $H_d$ the main hyperbolic component of $\mathcal{M}_d$:
$$H_d := \{ c \in \mathbb{C} : z^d +c \text{ has an attracting fixed point} \}.$$
See \cite{Douady-Hubbard} or \cite{Milnor} for more on hyperbolic components of $\mathcal{M}_d$ and fixed point theory.

\begin{prop} \label{sector} Let $d,n \ge 2$ be integers, let $\theta = \frac{1}{d^n-1}$ and $\theta' = \frac{d}{d^n-1}$.  Then $\theta$ and $\theta'$ form a landing pair, and their common landing point lies on the boundary of the main hyperbolic component $H_d$.
\end{prop} 

The proof follows the ideas of  \cite[Proposition~2.16]{Carminati-Isola-Tiozzo}  (which yields the case $d=2$); however, for $d > 2$ the combinatorics are more delicate, so we prove the Proposition for all $d \geq 2$. Before proceeding to the proof of Proposition~\ref{sector} we introduce the notation and state the necessary facts regarding the combinatorics of subsets of $S^1$ required for our arguments. We denote by $\tau_d$ the $d$-multiplication map on $\mathbb{R} / \mathbb{Z}$; by abuse of notation, we denote also by $\tau_d$ the induced map on $S^1$, i.e. $\tau_d\left(e^{2\pi\alpha}\right)=e^{2\pi d\alpha}$. Following the notation of \cite{Carminati-Isola-Tiozzo} and  \cite{Goldberg}, we make the following definitions.

\begin{defi} Given $r = \frac{m}{n} \in \mathbb{Q}$ (not necessarily in lowest terms), we say that a finite set $X \subset S^1$ is a degree $d$ $m/n$-rotation set if $\tau_d(X) = X$, and the restriction of $\tau_d$ to $X$ is conjugate to the circle rotation $R_r$ via an orientation-preserving homeomorphism of $S^1$.  Writing $\frac{m}{n} = \frac{p}{q}$ in lowest terms, we say $X$ has {\bf rotation number} $p/q$. 
\end{defi}

A useful equivalent definition is the following: $X = \{ \theta_i \} \subset \mathbb{R} / \mathbb{Z}$, indexed so that 
$$0 \leq \theta_0 < \theta_1 < \cdots < \theta_{n-1} < 1,$$ 
is a degree $d$ $m/n$-rotation subset of $S^1$ if it satisfies 
$$\tau_d(\theta_i) \equiv \theta_{i + m \mod n} \mod 1$$
for all $0 \leq i < n-1$.  Note that by Corollary 6 of \cite{Goldberg}, if $\frac{m}{n} = \frac{p}{q}$ is in lowest terms, any such $X$ has $n = kq$ elements, where $1 \leq k \leq d-1$.

\begin{defi} Given a set 
$$S = \{ \theta_0, \dots, \theta_{n-1} \}$$ 
satisfying 
$$0 \leq \theta_0 < \theta_1 < \cdots < \theta_{n-1} < 1,$$
suppose that $S$ is a degree $d$ $m/n$-rotation set. The {\bf deployment sequence} of $S$ is the ordered set of $d$ integers $\{ s_1, \dots, s_{d-1} \}$, where $s_i$ denotes the total number of angles $\theta_j$ in the interval $[ 0, \frac{i}{d-1} )$.
\end{defi}

By \cite[Theorem 7]{Goldberg}, rotation subsets are determined by the data of rotation number and deployment sequence, as shown in the following statement.

\begin{thm}[Goldberg \cite{Goldberg}] \label{goldberg}  A rotation subset of the unit circle is uniquely determined by its rotation number and its deployment sequence.  Conversely, given $r = \frac{p}{q}$ in lowest terms and a deployment sequence
$$0 \leq s_1 \leq s_2 \leq \cdots \leq s_{d-1} = kq,$$
there exists a rotation subset of $S^1$ with this rotation number and deployment sequence if and only if every class modulo $k$ is realized by some $s_j$.
\end{thm}

Now we can proceed to the proof of Proposition~\ref{sector}.
\begin{proof}[Proof of Proposition~\ref{sector}.]  
Let $r = \frac{1}{n}$.  The boundary of $H_d$ contains exactly $d-1$ parameters $c_i$ ($1\leq i\leq d-1$) so that $f_{c_i}$ has a fixed point of multiplier $e^{2 \pi i r}$.  For each $c_i$, denote by $\alpha_i$ the unique parabolic fixed point of $f_{c_i}$, which necessarily attracts the unique non-fixed critical point 0 (see \cite{Milnor}).  Consequently, the Julia set of $f_{c_i}$ (which we denote by  $J_{f_{c_i}}$) is locally connected, so all external rays land, and we have a landing map $L_i: \mathbb{R} / \mathbb{Z} \rightarrow J_{f_{c_i}}$, with
$$L_i(\theta) := \lim_{\rho \rightarrow 1^+} \phi_{f_{c_i}}^{-1} (\rho e^{2 \pi i \theta}),$$
where $\phi_{f_{c_i}}$ is the normalized uniformization map $\phi_{f_{c_i}}: \mathbb{C} \setminus K_{f_{c_i}} \rightarrow \mathbb{C} \setminus \overline{\mathbb{D}}$.   Then the landing map defines a continuous semiconjugacy, so that 
$$f_{c_i} ( L_i(\theta)) = L_i(\tau_d(\theta)).$$

Define $S_i := L_i^{-1}(\alpha_i),$ the set of angles whose rays land at $\alpha_i$.  Since $f_{c_i}$ has multiplier $e^{2 \pi i / n}$ at $\alpha_i$, $S_i$ is a rotation subset of $S^1$ with rotation number $r = \frac{1}{n}$ (see Lemma 2.4 of \cite{Goldberg-Milnor}). Writing 
$$S_i = \{ 0 \leq \theta_{i, 0} < \theta_{i, 1} < \cdots < \theta_{i, kn-1} < 1 \},$$
with $k$ cycles, each pair $\left(\theta_{i, j\mod kn}, \theta_{i,j+1 \mod kn}\right)$ cuts out an open arc $a_{i,j}$ in $S^1$.  Since the map preserves cyclic orientation and the collection $S_i$, any arc $a_{i,j}$ with length less than $\frac{1}{d}$ is mapped homeomorphically onto another arc and expands by a factor of $d$.  Since $n > 1,$ no arc is fixed, so for each arc, some iterate of the arc cannot be mapped homeomorphically onto its image, and has length $\geq \frac{1}{d}$.  For such an arc, the map $\tau_d$ maps the arc onto $S^1$, and so the sector $S$ which is bounded by the corresponding rays in the dynamical plane has image containing the set
$$T:= \mathbb{C} \setminus \bigcup_{0 \leq j \leq kn-1} \mathcal{R}(\theta_{i, j}).$$
Since the critical point is contained in an attracting petal, both the critical point $0$ and the critical value $c_i$ are contained in the Fatou set, so in $T$.  Since $S$ maps onto $T$, $f_{c_i}(S)$ contains the critical value, and so $S$ contains the critical point.  Since the sectors are disjoint, we conclude that there is a unique orbit of arcs, and so a unique cycle of elements of $S_i$, i.e. $k = 1$.  Note also that the unique sector corresponding to an arc of length $\geq \frac{1}{d}$ is the only sector which can contain a fixed ray, and therefore each element $s_{i, j}$ of the deployment sequence satisfies $s_{i,j} \in \{ 0, n \}$. 

We know by \cite{Schleicher} that the angles of parameter rays landing on $c_i$ is a subset of $S_i$ for each $i$; by $\tau_d$-invariance, then, the $S_i$ are disjoint.  Therefore, by Theorem~\ref{goldberg}, the sets $S_1, \dots, S_{d-1}$ are precisely the $d-1$ rotation subsets with rotation number $r = \frac{1}{n}$ and  deployment sequences containing only the values $0$ and $n$.  Noting that the subset 
$$F := \left\{ \frac{1}{d^n-1}, \frac{d}{d^n-1}, \dots, \frac{d^{n-1}}{d^n-1} \right\}$$
has rotation number $\frac{1}{n}$ and deployment sequence $\{ n, n, \dots, n \}$, we see that there exists an $i$ such that $F$ is the set of angles of dynamic rays landing at $\alpha_i$ for $f_{c_i}$; without loss of generality, $F = S_1$.  

It remains to show that $\frac{1}{d^n-1}$ and $\frac{d}{d^n-1}$ land together on $c_1$ in parameter space.  The pair of angles in $S_1$ landing on $c_1$ in parameter space are precisely the angles of $S_1$ whose dynamic rays are {\it characteristic}; that is, their wake separates the critical point 0 from the critical value $c_1$ (see \cite{Schleicher}).  Since $\tau_d$ maps an arc $[ \frac{d^k}{d^n-1}, \frac{d^{k+1}}{d^n-1}] \subset \mathbb{R} / \mathbb{Z}$ homeomorphically onto its image if and only if $0 \leq k < n-1$, the wake in the dynamical plane of $\mathcal{R}(\frac{d^{n-1}}{d^n-1})$ and  $\mathcal{R}(\frac{d^n}{d^n-1}) = \mathcal{R}(\frac{1}{d^n-1})$ contains the critical point.  Therefore the wake of $\mathcal{R}(\frac{1}{d^n-1})$ and $\mathcal{R}(\frac{d}{d^n-1})$ contains the critical value, and the conclusion follows.
\end{proof}

\begin{proof}[Proof of Theorem \ref{MnotJ}]
Suppose $h$ is a polynomial with filled Julia set $K_h = \mathcal{M}_d$. It is clear that such a $h(z)$ must have $D=$ deg$(h)\geq 2$ (since the filled Julia set of a linear polynomial is either the empty set, or a point, or the whole complex plane).  Therefore we have a B\"ottcher coordinate $\phi_h$ on $\mathbb{C} \setminus \mathcal{M}$ as described above. 

\begin{lemma} Suppose that $h(z)$ is a polynomial of degree $D \geq 2$ with $K_h = \mathcal{M}_d$.  Then the B\"ottcher coordinate $\phi_h$ can be chosen to be the map $\Phi: \mathbb{C} \setminus \mathcal{M}_d \rightarrow \mathbb{C} \setminus \overline{\mathbb{D}}$, and $h$ can be chosen to be monic.  
\end{lemma}

\begin{proof} Since $\mathcal{M}_d$ is connected, any such polynomial will have B\"ottcher coordinate which extends to the full $\mathbb{C} \setminus K_h$. Therefore $\phi_h(\Phi^{-1}(z))$ is an automorphism of $\mathbb{C} \setminus \overline{\mathbb{D}}$ which fixes $\infty$; consequently, it is a rotation by some $\lambda$ on the unit circle. Thus $\phi_h(z) = \lambda \Phi(z)$ for all $z \in \mathbb{C} \setminus \mathcal{M}_d$.  Note then that
$$\lim_{z \rightarrow \infty} \frac{\phi_h(z)}{z} = \lambda,$$
so $\lambda^{D-1}$ is the leading coefficient of $h$, call it $a_D$. 

Denote by $\overline{h}(z)$ the polynomial whose coefficients are complex conjugates of those of $h$. Since $\mathcal{M}_d$ is invariant under complex conjugation, then $\mathcal{M}_d$ is also the filled Julia set of $\overline{h}(z)$.  By the classification theorems of polynomials with the same Julia set (see \cite{Schmidt-Steinmetz}, noting that $J_h = \partial \mathcal{M}_d$ cannot be a line segment or a circle), some iterates $h^k$ and $\overline{h}^{\ell}$ then differ by an affine symmetry of $\mathcal{M}_d$; by degree, $k = \ell$, and we may replace $h$ by $h^k$ to assume $h$ and $\overline{h}$ differ by an affine symmetry of $\mathcal{M}_d$.  By Proposition \ref{affinesymm}, the affine symmetry group of $\mathcal{M}_d$ is the set of rotations by $(d-1)$-st roots of unity $\{ \zeta_{d-1}^k : 0 \leq k < d-1 \}$, we have $h(z) = \zeta_{d-1}^k \overline{h}(z)$ for all $z$.  Therefore the leading coefficient $a_D$ of $h(z)$ has the property that $a_D^{d-1} \in \mathbb{R}$; replacing $h(z)$ by the appropriate $\zeta_{d-1}^{\ell} h(z)$ leaves the Julia set invariant, and has real leading coefficient.  Since the new $a_D$ also has modulus 1, it is $\pm 1$; if $d$ is odd, we can choose $\zeta_{d-1}^k$ so that $h$ is monic. 

Suppose that $d$ is even and $a_D = \pm 1$. The intersection $\M_d\cap\R$ is an interval $[\alpha, \beta]$ with $\alpha<0$ and $\beta>0$. The interval $(0,\beta)$ is contained in the interior of $\M_d$. We let $\{\alpha_n\}_n$ be a sequence of  centers of hyperbolic components of $\M_d$ such that $\lim_{n\to\infty}\alpha_n=\alpha$; note that no $\alpha_n$ is contained in the interior of $\M_d$. Now we claim that $h(\beta)=\beta$. Indeed, since both $\C\setminus\M_d$ and $\M_d$ are invariant under $h$, we conclude that $h(\beta)$ is on the boundary of $\M_d$ (since $\beta\in\partial M_d$); so, if $h(\beta)\ne \beta$, then $h(\beta)<0$. Moreover, because $h(x)\notin \M_d$ for any $x>\beta$, we conclude that $h(\beta)=\alpha$. But then $h((0,\beta))$ is an interval contained in the interior of $\M_d$, whose closure contains $\alpha$; thus $h((0,\beta))$ contains an interval $(\alpha,\gamma)$ with $\alpha<\gamma<0$. So, $h((0,\beta))$ is contained in the interior of $\M_d$ but contains infinitely many points $\alpha_n$, which are not in the interior of $\M_d$. This contradiction shows that indeed $h(\beta)=\beta$. Hence $h((\beta,+\infty))$ is an infinite interval containing  $\beta$; thus $\lim_{x\to +\infty}h(x)=+\infty$ which yields that $a_D$ cannot be negative, and so, $a_D=1$, as desired.
\end{proof}

It follows that we have the following commutative diagram:
\begin{center} 

\begin{tikzpicture}
  \matrix (m) [matrix of math nodes,row sep=3em,column sep=4em,minimum width=2em]
     {
     \mathbb{C} \setminus \mathcal{M}_d & \mathbb{C} \setminus \overline{\mathbb{D}}  \\
      \mathbb{C} \setminus \mathcal{M}_d & \mathbb{C} \setminus \overline{\mathbb{D}} \\};
  \path[-stealth]
    (m-1-1) edge node [left] {$h$} (m-2-1)
            edge node [above] {$\Phi$} (m-1-2)
   (m-1-2) edge node [right] {$z \mapsto z^D$} (m-2-2)
    (m-2-1) edge node [below] {$\Phi$} (m-2-2);
           
\end{tikzpicture}
\end{center}

Since $\mathcal{R}(\frac{1}{d^m-1})$ and $\mathcal{R}(\frac{d}{d^m-1})$ land together for all $m \geq 2$, so do their forward iterates under $h$; that is, $\mathcal{R}(\frac{D^k}{d^m-1})$ and $\mathcal{R}(\frac{dD^k}{d^m-1})$ land together for all $m \geq 2$, $k \geq 0$.  Recall the following classical result of Bang \cite{Bang}.

\begin{thm} [Bang \cite{Bang}] Let $a \geq 2$ be an integer.  There exists a positive integer $M$ such that for each integer $m > M$, the number $a^m-1$ has a primitive prime divisor; that is, a prime $p$ such that $p \mid (a^m-1)$ and $p \nmid (a^k-1)$ for all $k < m$. 
\end{thm}

According to Bang's theorem, choose an integer $M >2$ such that  for all $m \geq M$, $d^m-1$ has a primitive prime divisor $p$ which does not divide $D$.  For any $k$, there exists a unique integer $m(k)$ such that
$$d^{m(k)}-1 \leq D^k(d^2-1) < d^{m(k)+1}-1.$$
We choose $K$ sufficiently large so that for all $k \geq K$, then $m(k) \geq M$.  For any $k \geq K$, write $m := m(k)$; as noted above, the parameter rays $\mathcal{R}(\frac{D^k}{d^m-1})$ and $\mathcal{R}(\frac{dD^k}{d^m-1})$ land together, say at $c$.  Since these are periodic rays with period dividing $m$, $f_c(z)$ has a periodic point $\alpha$ of period $n\mid m$ at which the dynamic rays of the same angles land (see \cite{Douady-Hubbard} for a proof of this).  Therefore $\frac{D^k}{d^m-1}$ and $\frac{dD^k}{d^m-1}$ lie in the same cycle of $\mathbb{R} / \mathbb{Z}$ under multiplication by $d^n$; note this cycle has maximal length $m/n$.  So for some $0 \leq r < m/n,$
$$\frac{D^k}{d^m-1} \equiv d^{rn+1} \frac{D^k}{d^m-1} \mod 1,$$
and so 
$$D^k(1-d^{rn+1}) \equiv 0 \mod (d^m-1).$$
Choose a primitive prime divisor $p$ of $d^m-1$ which is coprime to $D$.  Primitivity and the congruence above then imply $rn+1 = m$ (also note that $r<m/n$).  Since $n \mid m$, the only possibility is $n=1$.  So $\alpha$ is a fixed point, and $c$ must lie on the main hyperbolic component $H_d$.  However, by choice of $m = m(k)$ and the fact that $(d-1)$ divides $d^2-1$ and $d^{m+1}-1$, we have
$$\frac{1}{d^2-1} \leq \frac{D^k}{d^m-1} \leq \frac{d}{d^2-1}.$$
Since $\mathcal{R}(\frac{1}{d^2-1})$ and $\mathcal{R}(\frac{d}{d^2-1})$ land together on the main hyperbolic component, any ray in their wake cannot possibly do the same.  We conclude that $D^k = \frac{d^m-1}{d^2-1}$ or $D^k = d \cdot \frac{d^m-1}{d^2-1}$, which is impossible since $p \mid (d^m-1)$ and $p \nmid D\cdot (d^2-1)$.  Therefore we have a contradiction, and Theorem \ref{MnotJ} is proved. 
\end{proof}

\section{Unlikely Algebraic Relations Among Certain Sets in Dynamics}
\label{algebraic}

In this Section, we provide a second proof of Theorem \ref{MnotJ} (see Corollary~\ref{MnotJ corollary}), namely that for each $d\ge 2$, there exists no polynomial $h(z)$ of degree at least $2$ such that $\M_d$ is the filled Julia set $K_h$ of $h$. Our method is motivated by some natural questions (see Questions~\ref{MnotJ conjecture} and \ref{q:1st}) in complex dynamics which can be viewed as examples of \emph{unlikely relations} between sets associated to algebraic dynamical systems (such as the $d$-th multibrot set, or the Julia set or Fatou components). In each case the expectation is that a polynomial relation between two Julia sets, or two Fatou components, or a polynomial relation on the $d$-th multibrot set is induced by a linear relation between those sets.  Firstly, we note that arguing identically as in the proof of Theorem~\ref{main result} we can strengthen  Theorem~\ref{MnotJ} as follows.

\begin{thm}
\label{MnotJ 2}
There exists no polynomial $h\in\C[z]$ of degree greater than $1$ such that $h^{-1}(\M_d)=\M_d$. 
\end{thm}

\begin{proof}
Since $\M_d$ is a bounded subset of $\C$, we get that $\M_d\subseteq K_h$. On the other hand, since $J_h$ is the smallest closed subset of $\C$ totally invariant under $h$ which contains at least $3$ points, we conclude that $J_h\subset \M_d$. Since $\C\setminus \M_d$ is connected and $M_d$ is closed, we conclude that $K_h\subseteq \M_d$ and therefore $\M_d=K_h$, which contradicts Theorem~\ref{MnotJ} (or Corollary~\ref{MnotJ corollary}).  
\end{proof}

We ask whether one could weaken the hypothesis of Theorem~\ref{MnotJ 2} even further.

\begin{question}
\label{MnotJ conjecture}
Let $d \ge 2$, and let $h\in\C[z]$ such that $h(\M_d)=\M_d$. Is it true that $h(z)$ must be a linear polynomial of the form $\xi z$ for some $(d-1)$-st root of unity $\xi$?
\end{question}

Essentially, Conjecture~\ref{MnotJ conjecture} predicts the polynomial relations in the parameter space. It is natural to ask what are the polynomial relations also in the dynamical space; so we formulate below two questions which were motivated by Theorem~\ref{MnotJ 2}.

\begin{question}\label{q:1st}
Let $P(z),f(z),g(z)\in C[z]$ be polynomials of degrees at least 2. 
\begin{itemize}
	\item [(a)] Assume $P(J_f)=J_g$. Is it true that $J_g$ is the image of $J_f$
	under a linear polynomial? 
	\item [(b)] Assume there is a Fatou component $U_f$ of $f$, and a Fatou component $U_g$ of $g$ such that $P(U_f)=U_g$. Is it true that $J_g$ is the image of $J_f$ under a linear polynomial?
\end{itemize} 
\end{question}

The philosophy behind Question~\ref{q:1st} is to ask that whenever there is a polynomial relation between the Julia sets of the polynomials $f$ and $g$, or between a Fatou component of $f$ and a Fatou component of $g$, this relation forces a strong rigidity on the dynamical systems corresponding to $f$ and $g$. Indeed, the conclusion that $J_g=\mu(J_f)$ for some linear polynomial $\mu$ yields that $g$ and $\mu\circ f\circ \mu^{-1}$ have the same Julia sets; hence a polynomial relation between corresponding dynamical components of $f$ respectively of $g$ yields a geometric relation between the corresponding Julia sets. Note that Schmidt and Steinmetz \cite{Schmidt-Steinmetz} classified the polynomials with the same Julia sets; apart from the exceptional classes of monomial mappings and the Chebyshev polynomials, any two polynomials with the same Julia set must have a common iterate. Hence there is a rigid relation between the two dynamical systems of $f$ and of $g$ knowing that there is a polynomial relation  between either the two Julia sets or between two Fatou components. We also note that Question~\ref{q:1st} (a) is connected with the Dynamical Manin-Mumford Conjecture (see \cite{Zhang-lec} and \cite{IMRN}) which, in a special case,  asks for classifying the polynomial relations between the sets of preperiodic points of two rational maps. Motivated by Question~\ref{q:1st} we found an alternative proof of Theorem~\ref{MnotJ} (independent from the one presented in Section~\ref{Mandelbrot section}), which in particular answers Question~\ref{q:1st} positively if $f$ is a monomial. 

\begin{thm}\label{thm:1st}
Let $P(z)$, $h(z)\in\C[z]$ with degree at least 2.  Suppose that either $P(S^1) = J_h$, or $P(\mathbb{D})$ is a Fatou component of $h$.  Then $J_h$ is a circle.
\end{thm}

Theorem \ref{thm:1st} gives a positive answer to Question \ref{q:1st} when $f$ is linearly conjugate to a monomial (i.e.~$J_f$ is a circle), and suggests a general approach to be pursued for an answer to Question \ref{q:1st}.  We have the following immediate corollary:
\begin{cor}
\label{MnotJ corollary}
There is no polynomial $h(z)\in\C[z]$ having degree at least 2 such that $K_h=\cM_d$.
\end{cor}
\begin{proof}
Assume there is such an $h$. Recall the main hyperbolic component:
$$H_d := \{ c \in \mathbb{C} : z^d +c \text{ has an attracting fixed point} \}.$$
Let $\alpha$ be an attracting fixed point with multiplier $\lambda$
satisfying $|\lambda|<1$. From $d\alpha^{d-1}=\lambda$, and $c=\alpha-\alpha^d$,
we have that $H_d$ is exactly the image of the disk of radius $(\frac{1}{d})^{1/(d-1)}$ centered at the origin
under the polynomial $z-z^d$. Since $H_d$ is a connected component of the interior of $\cM_d$, it is a Fatou component of $h$. Theorem \ref{thm:1st} implies
that $J_h=\partial \cM_d$ is a circle, a contradiction.
\end{proof}

We will prove Theorem \ref{thm:1st} by studying the set of solutions $(u,w)\in S^1 \times S^1$ of equations of the form $r\circ q(u)=q(w)$, where $r,q\in \C(z)$ are non-constant. Using a simple trick of Quine \cite{Quine}, one embeds $r(S^1)$ into an algebraic curve in $\bP^1(\C) \times \bP^1(\C)$ as follows. Define
$$\eta:\ \bP^1(\C)\rightarrow \bP^1(\C) \times \bP^1(\C)$$
by $\eta(z)=(z,\bar{z})$ for $z\in\C$, and $\eta(\infty)=(\infty,\infty)$. Let $(x,y)$ denote the coordinate function on $\bP^1 \times \bP^1$. Then $\eta(S^1)$ is contained in the closed curve $C$ defined by $xy=1$. For any
non-constant rational map $r(z)$, let $C_r$ denote the closed curve:
$$C_r:=\{(r(x),\bar{r}(y)):\ (x,y)\in C\}.$$
The key observation is that $\eta(r(S^1))$ is contained in $C_r$.  Note that $C_r$ is the image of $C$ under the self-map $(r,\bar{r})$ of $\bP^1 \times \bP^1$, so $C_r$ is irreducible. 

We have the following simple lemma:
\begin{lemma}\label{lem:invariant curve}
Let $r(z),q(z)\in \C(z)$ be non-constant. Assume the equation $r\circ q(u)=q(w)$
has infinitely many solutions $(u,w)\in (S^1)^2$. Then the curve
$C_q$ is invariant under the self-map $(r,\bar{r})$ of $\bP^1 \times \bP^1$. 
\end{lemma}
\begin{proof}
The given assumption implies that the irreducible curves $C_q$ and $C_{r\circ q}$
have an infinite intersection. Hence $C_q=C_{r\circ q}$. Note that $C_{r\circ q}$
is exactly the image of $C_q$ under $(r,\bar{r})$.
\end{proof}

If $r(z)\in \C[z]$ is a polynomial of degree $D\geq 2$, then we may use the results of Medvedev and Scanlon from \cite{Medvedev-Scanlon} to classify the curves invariant under the map $(r, \bar{r})$ of $\bP^1 \times \bP^1$.  To do so, we require the following definition: 

\begin{defi} \label{exceptional} We call a polynomial $f(z)$ {\bf exceptional} if $f$ is conjugate by a linear map to either $\pm C_D(z)$ for a Chebyshev polynomial $C_D(z)$, or a powering map $z \mapsto z^D$.
\end{defi} 
We remind the readers that the Chebyshev polynomial
of degree $D$ is the unique polynomial $C_D(z)$ of degree $D$ such that
$C_D(z+\frac{1}{z})=z^D+\frac{1}{z^D}$. Note that by the classification theory of Julia sets in \cite{Schmidt-Steinmetz}, the Julia set of a polynomial is a line segment or a circle if and only if the polynomial is exceptional.  

\begin{prop}\label{prop:use Medvedev-Scanlon}
	Let $r(z)\in \C[z]$ be a non-exceptional polynomial of degree $D\geq 2$, and let $q(z)\in \C[z]$ be 
	non-constant. Then the curve $C_q$ is not invariant under $(r,\bar{r})$.
\end{prop}
\begin{proof}
 We assume that $C_q$ is invariant under $(r,\bar{r})$: note that $C_q$ is not a vertical or horizontal line. By Theorem 6.26 of \cite[p.~166]{Medvedev-Scanlon}, there exist polynomials $\pi(z)$, $\rho(z)$, $H(z)$
 and a curve $B$ in $\bP^1 \times \bP^1$ satisfying the following conditions:
 \begin{itemize}
 	\item [(i)] $r\circ \pi=\pi\circ H$, $\bar{r}\circ \rho=\rho\circ H$.
 	\item [(ii)] $C_r$ is the image of $B$ under the self-map $(\pi,\rho)$ of $\mathbb{P}^1 \times \mathbb{P}^1$. 		\item [(iii)] $B$ is periodic under the self-map $(H,H)$ of $\bP^1 \times \bP^1$. Furthermore, there is a non-constant polynomial $\psi(z)\in \C[z]$ commuting with an iterate of $H(z)$ such that $B$ is defined by the equation $y=\psi(x)$
 	or $x=\psi(y)$.
 \end{itemize}
 
 Assume $B$ is defined by $y=\psi(x)$ (the case of the equation $x=\psi(y)$
 can be treated similarly). Since $C_q$ is the image of $B$ under $(\pi,\rho)$, we have:
 $$C_q=\{(\pi(x),\rho(\psi(x))):\ x\in \bP^1(\C)\}.$$
Thus $(\infty,\infty)$ is the only point in $C_q$ whose first coordinate is $\infty$.
 
 On the other hand, we recall the definition of $C_q$:
 $$C_q:=\{(q(x),\bar{q}(y)):\ (x,y)\in C\}=\{(q(x),\bar{q}(1/x)):\ x\in \bP^1(\C)\}$$
 since $C$ is the closed curve defined by $xy=1$. So $(\infty,\bar{q}(0))$
 is the only point in $C_q$ whose first coordinate is $\infty$, a contradiction.
\end{proof}

\begin{proof} [Proof of Theorem \ref{thm:1st}] Let $D=\deg(g)\geq 2$ (note also that Question~\ref{q:1st} is trivial if $g$ were a linear polynomial). We first suppose $P(S^1)=J_g$. It follows that $g\circ P(S^1)=g(J_g)=J_g=P(S^1)$. By Lemma
	\ref{lem:invariant curve}, the curve $C_P$ is invariant under
	the self-map $(g,\bar{g})$ of $\bP^1 \times \bP^1$. By Proposition \ref{prop:use Medvedev-Scanlon}, $g$ must be linearly conjugate to $\pm C_D(z)$ or $z^D$, and so $J_g$ is a line segment or circle.  The image of $S^1$ under a polynomial cannot be a line segment. To see why, assume there were a polynomial $Q(z)$ mapping $S^1$ onto the interval $[0,1]$. The imaginary part of $Q$, which is harmonic and vanishes on $S^1$, must vanish on $\bar{\mathbb{D}}$ by the maximum modulus principle. Hence $Q(\bar{\mathbb{D}})\subseteq \R$, a contradiction since $Q(\bar{\mathbb{D}})$ has a non-empty interior. So $J_g$ is a circle, as desired.
	
	We now assume that $P(\mathbb{D})$ is a Fatou component of $g$. Because $P$ is a polynomial, we conclude that $P(\mathbb{D})$ is a bounded Fatou component.  By the No-Wandering-Domain Theorem of Sullivan \cite{Sullivan}, there exist $m\geq 0$ and $n>0$ such that $g^{m+n}(P(\mathbb{D}))=g^m(P(\mathbb{D}))$. Replacing $g$ by an iterate, we may assume $g^2(P(\mathbb{D}))=g(P(\mathbb{D}))$.  It follows that 
$$g^2\circ P(\bar{\mathbb{D}})=g\circ P(\bar{\mathbb{D}}).$$	

Note that for every open, bounded subset $S$ of $\C$, $\bar{S}-S$ is infinite: translating, we may assume $0\in S$, and then every ray originating from 0 will contain a point in $\bar{S}-S$. 
	
	Now let $S:=g\circ P(\mathbb{D})=g^2\circ P(\mathbb{D})$. Since $\bar{S}=g\circ P(\bar{\mathbb{D}})=g^2\circ P(\bar{\mathbb{D}})$, we have that $g\circ P(S^1)\cap g^2\circ P(S^1)$ contains the infinite set $\bar{S}-S$. Applying Lemma \ref{lem:invariant curve}, we see that the curve $C_{g\circ P}$ is invariant under $(g,\bar{g})$. By Proposition \ref{prop:use Medvedev-Scanlon}, $g$ is linearly conjugate to $\pm C_D(z)$ or $z^D$. Note that $J_g$ cannot be a line segment since the only Fatou component of such a $g$ is unbounded. Therefore $J_g$ is a circle, completing the proof of Theorem \ref{thm:1st}. 
\end{proof}




\end{document}